\documentclass{amsart}
\usepackage{amsmath}
\usepackage{amssymb}
\usepackage{amsfonts}

\newtheorem{theorem}{Theorem}
\theoremstyle{plain}

\newtheorem{corollary}{Corollary}

\newtheorem{definition}{Definition}
\newtheorem{example}{Example}

\newtheorem{lemma}{Lemma}

\newtheorem{proposition}{Proposition}
\newtheorem{remark}{Remark}

\numberwithin{equation}{section}
\input{tcilatex}

\begin{document}
\title{A Generalized Wintgen inequality for Legendrian submanifolds in
almost Kenmotsu manifolds}
\author{R.G\"{O}R\"{U}N\"{U}\c{S}}
\address{Uluda\u{g} University,Institute of Science, Bursa -TURKEY}
\email{rukengorunus16@gmail.com}
\author{ I. K\"{U}PELI ERKEN}
\address{Faculty of Natural Sciences, Architecture and Engineering,
Department of Mathematics, Bursa Technical University, Bursa-TURKEY}
\email{irem.erken@btu.edu.tr}
\author{Aziz YAZLA}
\address{Uludag University, Science Institute, Gorukle 16059, Bursa-TURKEY}
\email{501411002@ogr.uludag.edu.tr}
\author{C. MURATHAN}
\address{Uluda\u{g} University, Faculty\ of\ Art\ and\ Sciences, Department\
of\ Mathematics, Bursa -TURKEY}
\email{cengiz@uludag.edu.tr}

\begin{abstract}
Main interest of the present paper is to obtain the generalized Wintgen
inequality for Legendrian submanifolds in almost Kenmotsu manifolds.
\end{abstract}

\maketitle

\section{Introduction}

One of the most fundamental problems in a Riemannian submanifold theory is
to establish a simple sharp relationships between intrinsic and extrinsic
invariants. The main extrinsic invariants are the extrinsic normal
curvature, the squared mean curvature and the main intrinsic invariants
include the Ricci curvature and the scalar curvature. In 1979, Wintgen \cite%
{wintgen} obtained a basic inequality involving Gauss curvature $K$, normal
curvature $K^{\perp }$ and the squared mean curvature $\left \Vert
H\right
\Vert ^{2}$ of an oriented surface $M^{2}$ in $E^{4}$ , that is ,

\begin{equation}
K\leq \left \Vert H\right \Vert ^{2}-\left \vert K^{\perp }\right \vert
\label{WINTGEN}
\end{equation}%
with the equality holding if and only if the ellipse of curvature of $M^{2}$
in $E^{4}$ is a circle. The inequality (\ref{WINTGEN}), now called Wintgen
inequality, attracted the attention of several authors.

Over time P. J. De Smet, F. Dillen, L. Verstraelen and L. Vrancken \cite%
{Brno} \ gave a conjecture for Wintgen inequality in an $n$-dimensional
Riemannian submanifold \ $M^{n}$ of a real space form $R^{n+p}(c),$ namely, 
\begin{equation}
\rho \leq \left \Vert H\right \Vert ^{2}-\rho ^{\perp }+c,  \label{DDVV}
\end{equation}%
where%
\begin{equation}
\rho =\frac{2}{n(n-1)}\dsum \limits_{1\leq i<j\leq n}\left \langle
R(e_{i},e_{j})e_{j},e_{i}\right \rangle ,  \label{ro}
\end{equation}%
is the normalized scalar curvature of $M^{n}$%
\begin{equation}
\rho ^{\perp }=\frac{2}{n(n-1)}\sqrt{\dsum \limits_{1\leq i<j\leq n}\dsum
\limits_{1\leq \alpha <\beta \leq m}\left \langle R^{\perp
}(e_{i},e_{j})u_{\alpha },u_{\beta }\right \rangle ^{2},}  \label{ro dik}
\end{equation}%
where $\{e_{1},...,e_{n}\}$ and $\{u_{1},...,u_{p}\}$ respectively
orthonormal frames of \ tangent bundle $TM$ and normal bundle $T^{\perp }M$
and they also proved that this conjecture holds for codimension $p=2$. This
type of inequality known later came to be known as the DDVV conjecture. A
special version of the DDVV conjecture,%
\begin{equation}
\rho \leq \left \Vert H\right \Vert ^{2}+c,  \label{CHEN}
\end{equation}%
was proved by B.Y. Chen in \cite{chen}. F. Dillen, J. Fastenakels and J. Van
Der Vekens \cite{dillen2} proved that DDVV conjecture is equivalent to an
algebraic conjecture. Recently DDVV-conjecture was proven by Z. Lu \cite{lu}
and by Ge and Z. Tang \cite{Ge} indepently.

In recent years, I. Mihai \cite{ion} proved DDVV conjecture for Lagrangian
submanifolds in complex space forms and obtained Wintgen inequality for
Legendrian submanifolds in Sasakian space forms (see \cite{ion1}). On the
other hand, the product spaces $S^{n}(c)\times 
\mathbb{R}
$ and $%
\mathbb{R}
\times H^{n}(c)$ are studied to obtain generalized Wintgen inequality by Q.
Chen and Q. Cui \cite{Qchen}. Then J. Roth \cite{ROTH} extended DDVV
inequality to submanifolds of warped product manifolds $%
\mathbb{R}
\times _{f}M^{n}(c)$. Furuhata et al. \cite{FURUHATA} studied on Kenmotsu
statistical manifolds and warped product. It is well known that a Kenmotsu
manifold is locally considered as the warped product of a Kaehler manifold
and a line.

Nowadays Wintgen inequality of statistical submanifolds in statistical
manifolds of constant curvature has been studied in \cite{AMM}, \cite{AYMM}
and \cite{AM}. The generalized Wintgen inequality for statistical
submanifolds of statistical warped product manifolds is proved in \cite{CMBS}

Motivated by the studies of the above author, in this study, we consider
generalized Wintgen inequality for Legendrian submanifolds in almost
Kenmotsu manifolds.

\section{Preliminaries}

\label{preliminaries}

An almost Hermitian manifold $(N^{2n},g,J)$ is a smooth manifold endowed
with an almost complex structure $J$ and a Riemannian metric $g$ compatible
in the sense

\begin{equation*}
J^{2}X=-X,\text{ }g(JX,Y)=-g(X,JY)
\end{equation*}%
for any $X,Y\in \Gamma (TN).$ The fundamental $2$-form $\Omega $ of an
almost Hermitian manifold is defined by%
\begin{equation*}
\Omega (X,Y)=g(JX,Y)
\end{equation*}%
for any vector fields $X,Y$ on $N$. \ For an almost Hermitian manifold $%
(N^{2n},g,J)$ with Riemannian connection $\nabla $, the fundamental $2$-form 
$\Omega $ \ and the Nijenhuis torsion of $J$,\ $N_{J}$ satisfy 
\begin{equation}
2g((\nabla _{X}J)Y,Z)=g(JX,N_{J}(Y,Z)+3d\Omega (X,JY,JZ)-3d\Omega (X,Y,Z)
\label{NJ1}
\end{equation}%
where $N_{J}(X,Y)=[X,Y]-[JX,JY]+J[X,JY]+J[JX,Y]$ (see \cite{yano})$.$An
almost Hermitian manifold is said to be an almost Kaehler manifold if its
fundamental form $\Omega $ is closed, that is, $d\Omega $ $=0.$ If $d\Omega $
$=0$ and $N_{J}=0,$ the structure is called Kaehler.Thus by (\ref{NJ1}), an
almost Hermitian manifold $(N,J,g)$ is Kaehler if and only if its almost
complex structure $J$ is parallel with respect to the Levi-Civita connection 
$\nabla ^{0}$, that is, $\nabla ^{0}J=0$ (\cite{yano}).

It is known that a Kahlerian manifold $N$ is of constant holomorphic
sectional curvature $c$ if and only if 
\begin{equation}
R(X,Y)Z=\frac{c}{4}(g(X,Z)Y-g(Y,Z)X+g(JX,Z)Y-g(JY,Z)JX+2g(JX,Y)JZ),
\label{HOLC}
\end{equation}
and is denoted by $N(c)$ (see \cite{yano}).

Let $M$ be a $(2n+1)$-dimensional differentiable manifold and $\phi $ is a $%
(1,1)$ tensor field, $\xi $ is a vector field and $\eta $ is a one-form on $%
M.$ If $\phi ^{2}=-Id+\eta \otimes \xi ,\quad \eta (\xi )=1$ then $(\phi
,\xi ,\eta )$ is called an almost contact structure on $M$ . The manifold $M$
is said to be an almost contact manifold if it is endowed with an almost
contact structure \cite{blair}.

Let $M$ be an almost contact manifold. $M$ will be called an almost contact
metric manifold if it is additionally endowed with a Riemannian metric $g$ ,
i.e.%
\begin{equation}
g(\phi X,\phi Y)=g(X,Y)-\eta (X)\eta (Y).  \label{1}
\end{equation}

For such manifold, we have 
\begin{equation}
\eta (X)=g(X,\xi ),\text{ }\phi (\xi )=0,\text{ }\eta \circ \phi =0.
\label{2}
\end{equation}

Furthermore, a $2$-form $\Phi $ is defined%
\begin{equation}
\Phi (X,Y)=g(\phi X,Y),  \label{3}
\end{equation}%
and usually \ is called fundamental form.

~On an almost contact manifold, the $(1,2)$-tensor field $N^{(1)}$ is
defined by%
\begin{equation*}
N^{(1)}(X,Y)=\left[ \phi ,\phi \right] (X,Y)-2d\eta (X,Y)\xi ,
\end{equation*}%
where $\left[ \phi ,\phi \right] $ is the Nijenhuis torsion of $\phi $%
\begin{equation*}
\left[ \phi ,\phi \right] (X,Y)=\phi ^{2}\left[ X,Y\right] +\left[ \phi
X,\phi Y\right] -\phi \left[ \phi X,Y\right] -\phi \left[ X,\phi Y\right] .
\end{equation*}

If $N^{(1)}$ vanishes identically, then the almost contact manifold
(structure) is said to be normal \cite{blair}. The normality condition says
that the almost complex structure $J$ defined on $M\times 
\mathbb{R}
$%
\begin{equation*}
J(X,\lambda \frac{d}{dt})=(\phi X+\lambda \xi ,\eta (X)\frac{d}{dt}),
\end{equation*}%
is integrable.

An almost contact metric manifold $M^{2n+1}$, with a\ structure $(\phi ,\xi
,\eta ,g)$ is said to be an almost cosymplectic manifold, if 
\begin{equation}
d\eta =0,\quad d\Phi =0.  \label{4}
\end{equation}%
If additionally normality conditon is fulfilled, then manifold is called
cosymplectic.

On the other hand, Kenmotsu studied in \cite{kenmotsu} another class of
almost contact manifolds, defined by the following conditions on the
associated almost contact structure%
\begin{equation}
d\eta =0,\quad d\Phi =2\eta \wedge \Phi .  \label{5}
\end{equation}%
A normal almost Kenmotsu manifold is said to be a Kenmotsu manifold.

\section{\protect \bigskip Statistical Manifolds}

Let $(M,g)$ be a Riemannian manifold and $\nabla $ an affine connection on $%
M $. An affine connection $\nabla ^{\ast }$ is said to be dual connection of 
$\nabla $ if%
\begin{equation}
Zg(X,Y)=g(\nabla _{Z}X,Y)+g(X,\nabla _{Z}^{\ast }Y)  \label{STATISTIC}
\end{equation}%
for any $X,Y,Z\in \Gamma (M).$The notion of \textquotedblleft conjugate
connection" is given an excellent survey by Simon (\cite{NOSI}, \cite{SI}).
In the triple $(g,\nabla ,\nabla ^{\ast })$ is called a dualistic structure
on $M$. It appears that $(\nabla ^{\ast })^{\ast }$ $=\nabla $. The
statistical model often forms a geometrical manifold so that the geometry of
a manifold plays an important role in statistics. The manifold structure of
statistical distributions was first introduced by Amari \cite{Amari} and
used by Lauritzen \cite{LA}.

/A statistical manifold $(M,\nabla ,g)$ is a Riemannian manifold $(M,g)$
endowed torsion free connection $\nabla $ such that the Codazzi equation%
\begin{equation}
(\nabla _{X}g)(Y,Z)=(\nabla _{Y}g)(X,Z)  \label{ST1}
\end{equation}%
holds for holds for any $X,Y,Z$ $\in $ $\Gamma (TM)$ (see \cite{Amari} ). If 
$(M,\nabla ,g)$ is a statistical manifold, so is $(M,\nabla ^{\ast },g)$.
For a statistical manifold $(M,g$,$\nabla $,$\nabla ^{\ast }$) the
difference $(1,2)$ tensor $\mathcal{K}$ of \ a torsion free affine
connection $\nabla $ and Levi-Civita connection $\nabla ^{0}$ is defined as%
\begin{equation}
\mathcal{K}_{X}Y=\mathcal{K}(X,Y)=\nabla _{X}Y-\nabla _{X}^{0}Y.  \label{K1}
\end{equation}%
Because of $\nabla $ and $\nabla ^{0}$ are torsion free, we have%
\begin{equation}
\mathcal{K}_{X}Y=\mathcal{K}_{Y}X,\text{ \  \ }\ g(\mathcal{K}_{X}Y,Z)=g(Y,%
\mathcal{K}_{X}Z)  \label{K2}
\end{equation}%
for any $X,Y,Z\in \Gamma (TM)$. By (\ref{STATISTIC}) and (\ref{K1}), one can
obtain 
\begin{equation}
\mathcal{K}_{X}Y=\nabla _{X}^{0}Y-\nabla _{X}^{\ast }Y.  \label{K3}
\end{equation}%
Using (\ref{K1}) and (\ref{K3}), we find 
\begin{equation}
2\mathcal{K}_{X}Y=\nabla _{X}Y-\nabla _{X}^{\ast }Y.  \label{K4}
\end{equation}%
By (\ref{K1}), we have 
\begin{equation}
g(\nabla _{X}Y,Z)=g(\mathcal{K}_{X}Y,Z)+g(\nabla _{X}^{0}Y,Z).  \label{K5}
\end{equation}%
It can \ be also shown that any torsion-free affine connection $\nabla $ has
a dual connection given by 
\begin{equation}
\nabla ^{0}=\frac{1}{2}(\nabla +\nabla ^{\ast }),
\end{equation}%
where $\nabla ^{0}$ is Levi-Civita connection of the Riemannian manifold $%
(M,g)$. If $\nabla =\nabla ^{\ast }$ then ($M,\nabla ,g)$ is called trivial
statistical manifold.

Denote by $R$ and $R^{\ast }$ the curvature tensors on $M$ with respect to
the affine connection $\nabla $ and its conjugate $\nabla ^{\ast }$,
respectively. Then the relation between $R$ and $R^{\ast }$ can be given as
following%
\begin{equation}
g(R(X,Y)Z,W)=-g(Z,R^{\ast }(X,Y)W)  \label{CONJUGATE2}
\end{equation}%
for any $X,Y,Z,W\in \Gamma (TM).$

By using (\ref{K1}) and (\ref{K3}), we have 
\begin{equation*}
R(X,Y)Z+R^{\ast }(X,Y)Z=2R^{0}(X,Y)Z+2[\mathcal{K},\mathcal{K}](X,Y)Z,
\end{equation*}%
where $[\mathcal{K},\mathcal{K}](X,Y)Z=[\mathcal{K}_{X},\mathcal{K}_{Y}]Z=%
\mathcal{K}_{X}\mathcal{K}_{Y}Z-\mathcal{K}_{Y}\mathcal{K}_{X}Z$ (see \cite%
{OPOZDA} ).The $[\mathcal{K},\mathcal{K}]$ is a curvature-like tensor.

In \cite{TO}, L.Todjihounde gave a method how to establish a dualistic
structure on the warped product manifold $.$If we adapt this method for $%
I\times _{f}N$, we have

\begin{proposition}[ \protect \cite{TO}]
\label{TO1} Let $(%
\mathbb{R}
,dt,\nabla )$ be a trivial statistical manifold and $(N,g_{N},^{N}\nabla $, $%
^{N}\nabla ^{\ast })$ be a statistical manifold . If $\ $the connections $%
\bar{\nabla}$ and $\bar{\nabla}^{\ast }\ $\ satisfy the following relations
on $%
\mathbb{R}
\times N$

(a) $\bar{\nabla}_{\bar{\partial}_{t}}\bar{\partial}_{t}=\nabla _{\partial
_{t}}\partial _{t}=0$

(b) $\bar{\nabla}_{\bar{\partial}_{t}}\bar{X}=\bar{\nabla}_{\bar{X}}\bar{%
\partial}_{t}=\frac{f^{\prime }(t)}{f(t)}X$

(c) $\bar{\nabla}_{\bar{X}}\bar{Y}$ $=$ $\ ^{N}\nabla _{X}Y-\frac{<X,Y>}{f}%
f^{\prime }(t)\partial _{t}$

and

(i) $\bar{\nabla}_{\bar{\partial}_{t}}^{\ast }\bar{\partial}_{t}=\nabla
_{\partial _{t}}^{\ast }\partial _{t}=0$,

(ii)$\bar{\nabla}_{\bar{\partial}_{t}}^{\ast }\bar{X}=\bar{\nabla}_{\bar{X}%
}^{\ast }\bar{\partial}_{t}=\frac{f^{\prime }(t)}{f(t)}X$

(iii) $\bar{\nabla}_{\bar{X}}^{\ast }\bar{Y}$ $=$ $\ ^{N}\nabla _{X}^{\ast
}Y-\frac{<X,Y>}{f}f^{\prime }(t)\partial _{t}$

then $(%
\mathbb{R}
\times _{f}N,<,>,\bar{\nabla},\bar{\nabla}^{\ast })$ is a statistical
manifold, where $\bar{X},\bar{Y}$ are vertical lifts of \ $X,Y\in \Gamma
(TN) $ and $\bar{\partial}_{t}$ $=\frac{\partial }{\partial t}$is horizontal
lift of $\partial _{t}$ and\ the notation is simplified by writing $f$ for $%
f\circ \pi $ and $\func{grad}f$ for $\func{grad}(f\circ \pi )$.$.$
\end{proposition}

Assuming $(%
\mathbb{R}
,dt,\nabla )$ is trival statistical manifold and denoting $R$ and $R^{\ast } 
$ are curvature tensors respect to the dualistic structure $(<,>,\bar{\nabla}%
,\bar{\nabla}^{\ast })$ on $%
\mathbb{R}
\times N$ \ then we can give the following lemma by Proposition \ref{TO1}.
In practise, $(-)$ is ommited from lifts.

\begin{lemma}[ \protect \cite{TO}]
\label{TO3} $(\tilde{M}=%
\mathbb{R}
\times _{f}N,<,>,\bar{\nabla},\bar{\nabla}^{\ast })$ is a statistical
manifold, be a statistical warped product. If $U,V,W\in \Gamma (N),$ then:

$(a)$ $R(V,\partial _{t})\partial _{t}=-\frac{f^{\prime \prime }(t)}{f(t)}V,$

$(b)$ $R(V,U)\partial _{t}=0,$

$(c)$ $R(\partial _{t},V)W=-\frac{f^{\prime \prime }(t)}{f(t)}<V,W>\partial
_{t},$

$(d)$ $R(V,W)U=R^{N}(V,W)U-\frac{(f^{\prime }(t))^{2}}{(f(t))^{2}}%
[<W,U>V-<V,U>W],$

and

$(a^{\ast })$ $R^{\ast }(V,\partial _{t})\partial _{t}=-\frac{f^{\prime
\prime }(t)}{f(t)}V,$

$(b^{\ast })$ $R^{\ast }(V,U)\partial _{t}=0,$

$(c^{\ast })$ $R^{\ast }(\partial _{t},V)W=-\frac{f^{\prime \prime }(t)}{f(t)%
}<V,W>\partial _{t},$

$(d^{\ast })$ $R^{\ast }(V,W)U=R^{\ast N}(V,W)U-\frac{(f^{\prime }(t))^{2}}{%
(f(t))^{2}}[<W,U>V-<V,U>W].$

where $R^{^{\ast }N}$ and $R^{N}$ are curvature tensors of $N$ with respect
to the connections $^{\text{\ }N}\nabla $ and $^{N}\nabla ^{\ast }$ .
\end{lemma}

\subsection{Statistical submanifolds}

In this section, we will give some basic notations, formulas, definitions
taken from reference \cite{VOS}.

Let $(M^{n},g)$ be a statistical submanifold of $(\tilde{M}^{n+d},<,>).$Then
The Gauss and Weingarten formulas are given respectively by%
\begin{equation}
\tilde{\nabla}_{X}Y=\nabla _{X}Y+h(X,Y),\  \  \tilde{\nabla}_{X}\xi =-A_{\xi
}X+D_{X}\xi ,,  \label{GAUSS1*}
\end{equation}%
\begin{equation}
\tilde{\nabla}_{X}^{\ast }Y=\nabla _{X}^{\ast }Y+h^{\ast }(X,Y),\text{ \  \ }%
\tilde{\nabla}_{X}^{\ast }\xi =-A_{\xi }^{\ast }X+D_{X}^{\ast }\xi ,,
\label{GAUSS2}
\end{equation}%
for $X,Y\in \Gamma (TM)$ and $\xi \in $ $\Gamma (T^{^{\bot }}M)$,
respectively. Furthermore, the following hold for $X,Y,Z\in \Gamma (TM)$ and 
$\xi ,\eta \in $ $\Gamma (T^{^{\bot }}M):$%
\begin{equation*}
Xg(Y,Z)=g(\nabla _{X}Y,Z)+g(Y,\nabla _{X}^{\ast }Z),
\end{equation*}%
\begin{equation*}
<h(X,Y),\xi >=g(A_{\xi }^{\ast }X,Y),\text{ \  \ }<h^{\ast }(X,Y),\xi
>=g(A_{\xi }X,Y)
\end{equation*}%
and 
\begin{equation*}
X<\xi ,\eta >=<\text{\  \ }D_{X}\xi ,\eta >+<\xi ,D_{X}^{\ast }\eta >.
\end{equation*}%
The mean curvature vector fields of $M$ are defined with respect to $\tilde{%
\nabla}$ and $\tilde{\nabla}^{\ast }$ by%
\begin{equation*}
H=\frac{1}{n}\sum \limits_{i=1}^{n}h(e_{i},e_{i})\text{ \ and \  \ }H^{\ast }=%
\frac{1}{n}\sum \limits_{i=1}^{n}h^{\ast }(e_{i},e_{i})\text{\ }
\end{equation*}%
where $\{e_{1},...,e_{n}\}$ is a local orthonormal frame of the tangent
bundle $TM$ of $M$. By (\ref{GAUSS1*}) and (\ref{GAUSS2}), we have $%
2h^{0}=h+h^{\ast }$ and $2H^{0}=H+H^{\ast },$ where $h^{0}$ and $H^{0}$ are
second fundamental form and mean curvature with respect to Levi-Civita
connection $\tilde{\nabla}^{0}$.

\begin{proposition}[\protect \cite{VOS}]
\label{Gauss} Let $(M^{m},g$,$\nabla ,\nabla ^{\ast })$ be statistical
submanifold of $(\tilde{M}^{m+n},<,>,\tilde{\nabla},\tilde{\nabla}^{\ast })$%
. Denote $\tilde{R}$ and $\tilde{R}^{\ast }$ the curvature tensors on $%
\tilde{M}^{m+n}$ with respect to connections $\tilde{\nabla}$ and $\tilde{%
\nabla}^{\ast }$. Then%
\begin{equation}
<\tilde{R}(X,Y)Z,W>=g_{M}(R(X,Y)Z,W)+<h(X,Z),h^{\ast }(Y,W)>-<h^{\ast
}(X,W),h(Y,Z)>,  \label{GAUSSEQU}
\end{equation}

\begin{equation}
<\tilde{R}^{\ast }(X,Y)Z,W>=g_{M}(R^{\ast }(X,Y)Z,W)+<h^{\ast
}(X,Z),h(Y,W)>-<h(X,W),h^{\ast }(Y,Z)>,  \label{DGAUSSEQU}
\end{equation}
\end{proposition}

\begin{equation}
<(R^{\bot }(X,Y)\xi ,\eta >=<\tilde{R}(X,Y)\xi ,\eta >+g_{M}([A_{\xi }^{\ast
},A_{\eta }]X,Y),  \label{RICCI}
\end{equation}%
\begin{equation}
<(R^{\ast \bot }(X,Y)\xi ,\eta >=<\tilde{R}^{\ast }(X,Y)\xi ,\eta
>+g_{M}([A_{\xi },A_{\eta }^{\ast }]X,Y),  \label{DRICCI}
\end{equation}

\textit{where }$R^{\bot }$\textit{\ and }$R^{\ast \bot }$\textit{\ are
curvature tensors with respect to }$D$\textit{\ and }$D^{\ast }$ and%
\begin{eqnarray*}
\lbrack A_{\xi },A_{\eta }^{\ast }] &=&A_{\xi }A_{\eta }^{\ast }-A_{\eta
}^{\ast }A_{\xi }, \\
\lbrack A_{\xi }^{\ast },A_{\eta }] &=&A_{\xi }^{\ast }A_{\eta }-A_{\eta
}A_{\xi }^{\ast }
\end{eqnarray*}%
\textit{for }$X,Y,Z,W\in \Gamma (TM)$\textit{\ and }$\xi ,\eta \in \Gamma
(T^{\bot }M)$.

\section{Almost Kenmotsu and cosymplectic statistical manifolds}

\begin{definition}[\protect \cite{FU}]
\label{FU}Let ($M,g,\nabla )$ be statistical manifold with almost complex
structure $J\in \Gamma (TM^{(1,1)})$. Denote by $\Omega $ the fundamental
form with respect to $J$ and $g$, that is, $\Omega (X,Y)=g(X,JY)$. The
triplet ($\nabla $,$g,J)$ is called a holomorphic statistical structure on M
if $\Omega $ is a $\nabla $-parallel 2-form.
\end{definition}

\begin{lemma}[\protect \cite{YAZLA}]
\label{FUNDAMENTAL FORM 2}For an almost Hermitian statistical manifold we
have%
\begin{equation}
(\nabla _{X}\Omega )(Y,Z)=g((\nabla _{X}J)Y,Z)-2g(\mathcal{K}_{X}JY,Z),
\label{AZIZ4}
\end{equation}%
and%
\begin{equation}
(\nabla _{X}^{\ast }\Omega )(Y,Z)=g((\nabla _{X}^{\ast }J)Y,Z)+2g(\mathcal{K}%
_{X}JY,Z)  \label{AZIZ5}
\end{equation}%
for any $X,Y,Z\in \Gamma (TM).$
\end{lemma}

\begin{corollary}[\protect \cite{YAZLA}]
\label{FUNDAMENTAL FORM ULUDAG}For an almost Hermitian statistical manifold
we have 
\begin{equation}
(\nabla _{X}\Omega )(Y,Z)=(\nabla _{X}^{0}\Omega )(Y,Z)-g(\mathcal{K}_{X}JY+J%
\mathcal{K}_{X}Y,Z)  \label{AZIZ5A}
\end{equation}%
and%
\begin{equation}
(\nabla _{X}^{\ast }\Omega )(Y,Z)=(\nabla _{X}^{0}\Omega )(Y,Z)+g(\mathcal{K}%
_{X}JY+J\mathcal{K}_{X}Y,Z)  \label{AZIZ5B}
\end{equation}%
for any $X,Y,Z\in \Gamma (TM).$
\end{corollary}

By Lemma \ref{FUNDAMENTAL FORM 2} and Corollary \ref{FUNDAMENTAL FORM ULUDAG}
we have following.

\begin{proposition}[\protect \cite{FU},\protect \cite{YAZLA}.]
Let ($M,g,\nabla ,J)$ be a holomorphic stattistical manifold and $\mathcal{K}%
_{X}JY+J\mathcal{K}_{X}Y=0$ for any $X,Y\in \Gamma (TM).$Then following
staments are equiavalent.

\begin{itemize}
\item $(M,g,\nabla ,J)$ is a holomorphic stattistical manifold,

\item $(M,g,\nabla ^{\ast },J)$ is a holomorphic stattistical manifold,

\item $(M,g,\nabla ^{0},J)$ is a Kaehler manifold.
\end{itemize}
\end{proposition}

\begin{definition}
Let ($M^{2n+1},g,\nabla ,\nabla ^{\ast })$ be a statistical manifold. If $%
M^{2n+1}$ $\ $is an almost contact metric manifold then $M^{2n+1}$ is called
almost contact metric statistical manifold.
\end{definition}

\begin{corollary}[\protect \cite{YAZLA}]
For an almost contact metric statistical manifold we have 
\begin{equation}
(\nabla _{X}\Phi )(Y,Z)=(\nabla _{X}^{0}\Phi )(Y,Z)-g(\mathcal{K}_{X}\phi
Y+\phi \mathcal{K}_{X}Y,Z)  \label{BB1}
\end{equation}%
and%
\begin{equation}
(\nabla _{X}^{\ast }\Phi )(Y,Z)=(\nabla _{X}^{0}\Phi )(Y,Z)+g(\mathcal{K}%
_{X}\phi Y+\phi \mathcal{K}_{X}Y,Z)  \label{BB2}
\end{equation}%
for any $X,Y,Z\in \Gamma (TM).$
\end{corollary}

\begin{proposition}[\protect \cite{YAZLA}]
\label{FUNDAMENTAL FORM 86}Let $(M^{n},g,\nabla ,\nabla ^{\ast })$ be a
statistical manifold and $\psi $ be a skew symmetric $(1,1)$ tensor field on 
$M$. \ Then we have%
\begin{equation}
g(\mathcal{K}_{X}\psi Y+\psi \mathcal{K}_{X}Y,Z)+g(\mathcal{K}_{Z}\psi
X+\psi \mathcal{K}_{Z}X,Y)+g(\mathcal{K}_{Y}\psi Z+\psi \mathcal{K}_{Y}Z,X)=0
\label{AZIZ IDENTITIY}
\end{equation}%
for any $X,Y,Z\in \Gamma (TM).$
\end{proposition}

If we resort to the relation (\ref{BB1}) and\ (\ref{AZIZ IDENTITIY}), we have%
\begin{eqnarray*}
(\nabla _{X}\Phi )(Y,Z)+(\nabla _{Z}\Phi )(X,Y)+(\nabla _{Y}\Phi )(Z,X)
&=&(\nabla _{X}^{0}\Phi )(Y,Z)+(\nabla _{Z}^{0}\Phi )(X,Y) \\
&&+(\nabla _{Y}^{0}\Phi )(Z,X)-2(g(\mathcal{K}_{X}\phi Y+\phi \mathcal{K}%
_{X}Y,Z) \\
&&+g(\mathcal{K}_{Z}\phi X+\phi \mathcal{K}_{Z}X,Y)+g(\mathcal{K}_{Y}\phi
Z+\phi \mathcal{K}_{Y}Z,X) \\
&=&(\nabla _{X}^{0}\Phi )(Y,Z)+(\nabla _{Z}^{0}\Phi )(X,Y)+(\nabla
_{Y}^{0}\Phi )(Z,X),
\end{eqnarray*}%
where $U,V,W\in \Gamma (TM).$

This relation shows clearly that 
\begin{eqnarray}
3d\Phi (X,Y,Z) &=&(\nabla _{X}^{0}\Phi )(Y,Z)+(\nabla _{Z}^{0}\Phi
)(X,Y)+(\nabla _{Y}^{0}\Phi )(Z,X)  \label{CONTACT5} \\
&=&(\nabla _{X}\Phi )(Y,Z)+(\nabla _{Z}\Phi )(X,Y)+(\nabla _{Y}\Phi )(Z,X). 
\notag
\end{eqnarray}

Let $(N,\nabla ,g,J)$ be an almost Hermitian manifold statistical manifold
and $(%
\mathbb{R}
,dt,^{%
\mathbb{R}
}\nabla )$ be trival statistical manifold. Let us consider the warped
product $\tilde{M}=%
\mathbb{R}
\times _{f}N$,with warping function $f>0,$endowed with the Riemannian metric

\begin{equation*}
<,>=dt^{2}+f^{2}g
\end{equation*}%
Denoting by $\xi =\frac{\partial }{\partial t}$ the structure vector field
on $\tilde{M}$, one can define arbitary any vector field on $\tilde{M}$ by $%
\tilde{X}=\eta (\tilde{X})\xi +X,$ where $X$ is any vector field on $N$ and $%
dt=\eta $. By the help of \ tensor field $J,$ a new tensor field $\phi $ of
type $(1,1)$ on $\tilde{M}$ can be given by 
\begin{equation}
\phi \tilde{X}=JX,\text{ \ }X\in \Gamma (TN),  \label{CONTACT3}
\end{equation}%
for $\tilde{X}\in $ $\Gamma (T\tilde{M})$. So we get $\phi \xi =0$, $\eta
\circ \phi =0,$ $\phi ^{2}\tilde{X}=-\tilde{X}+\eta (\tilde{X})\xi $ and $%
<\phi \tilde{X},\tilde{Y}>=-<\tilde{X},\phi \tilde{Y}>$ for $\tilde{X},%
\tilde{Y}\in \Gamma (T\tilde{M})$. Furthermore, we have $<\phi \tilde{X}%
,\phi \tilde{Y}>=<\tilde{X},\tilde{Y}>-\eta (\tilde{X})\eta (\tilde{Y})$
.Thus $(\tilde{M},<,>,\phi ,\xi ,\eta )$ is an almost contact metric
manifold. By Proposition \ref{TO1} and similar argument as in \cite{CA} we
have

\begin{equation}
(\tilde{\nabla}_{\tilde{X}}\phi )\tilde{Y}=(\nabla _{X}J)Y-\frac{f^{\prime
}(t)}{f(t)}<\tilde{X},\phi \tilde{Y}>\xi -\frac{f^{\prime }(t)}{f(t)}\eta (%
\tilde{Y})\phi \tilde{X}  \label{CONTACT4}
\end{equation}

Using Proposition \ref{TO1} \ we get 
\begin{equation*}
\mathcal{\tilde{K}}_{X}Y=\mathcal{K}_{X}Y,\mathcal{\tilde{K}}_{X}\xi =%
\mathcal{\tilde{K}}_{\xi }X=0,\mathcal{\tilde{K}}_{\xi }\xi =0,
\end{equation*}%
where $\mathcal{K}_{X}Y=\nabla _{X}Y-\nabla _{X}^{0}Y$ and $\mathcal{\tilde{K%
}}_{X}Y=\tilde{\nabla}_{X}Y-\tilde{\nabla}_{X}^{0}Y.$

By (\ref{CONTACT4}) and \ (\ref{CONTACT5}) it is readily found that the
relation%
\begin{eqnarray*}
(\tilde{\nabla}_{\tilde{X}}\Phi )(\tilde{Y},\tilde{Z}) &=&f^{2}(\nabla
_{X}\Omega )(Y,Z)-\frac{f^{\prime }(t)}{f(t)}<\tilde{X},\phi \tilde{Y}>\eta (%
\tilde{Z}) \\
&&-\frac{f^{\prime }(t)}{f(t)}\eta (\tilde{Y})\Phi (\tilde{X},\tilde{Z})
\end{eqnarray*}%
We thus conclude that%
\begin{equation}
d\Phi =f^{2}d\Omega +2(-\frac{f^{\prime }(t)}{f(t)})\eta \wedge \Phi
\label{CONTACT6}
\end{equation}%
and Proposition \ref{TO1} leading to the following theorem.

\begin{theorem}
Let $(%
\mathbb{R}
,dt,^{%
\mathbb{R}
}\nabla )$ be trival statistical manifold.\ Then the warped product $\tilde{M%
}=%
\mathbb{R}
\times _{f}N$ is an almost $(-\frac{f^{\prime }(t)}{f(t)})-$Kenmotsu
statistical manifold if and only if $(N,\nabla ,g,J)$ is an almost Kaehler
statistical manifold. Moreover $\mathcal{\tilde{K}}_{X}Y=\mathcal{K}_{X}Y,%
\mathcal{\tilde{K}}_{X}\xi =\mathcal{\tilde{K}}_{\xi }X=0,\mathcal{\tilde{K}}%
_{\xi }\xi =0,$where $\mathcal{K}=\nabla -\nabla ^{g}$, and $\mathcal{\tilde{%
K}=}\tilde{\nabla}-\tilde{\nabla}^{<,>}.$
\end{theorem}

Chossing $f=const\neq 0$ we have following corollary.

\begin{corollary}
Let $(%
\mathbb{R}
,dt,^{%
\mathbb{R}
}\nabla )$ be trival statistical manifold. Then the product manifold $\tilde{%
M}=%
\mathbb{R}
\times N$ is an almost cosymplectic statistical manifold if and only if $%
(N,\nabla ,g,J)$ is an almost Kaehler statistical manifold
\end{corollary}

Using same methods as in \cite{LO}, we get following proposition

\begin{proposition}
\label{PRP}Let $\tilde{M}=I\times _{f}N(c)$ be statistical warped product
manifold and $\tilde{X},\tilde{Y},\tilde{Z},\tilde{W}\in \Gamma (\tilde{M}),$
where $I\subset 
\mathbb{R}
$ is trival statistical manifold and $N(c)$ is statistical compleex space
form. Then the curvature tensor $\tilde{R}$ and $\tilde{R}^{\ast }$ are
given by 
\begin{eqnarray*}
\tilde{R}(\tilde{X},\tilde{Y},\tilde{Z},\tilde{W}) &=&\tilde{R}^{\ast }(%
\tilde{X},\tilde{Y},\tilde{Z},\tilde{W})=[\frac{c}{4f^{2}}-\frac{(f^{\prime
})^{2}}{f^{2}}][<\tilde{Y},\tilde{Z}><\tilde{X},\tilde{W}>-<\tilde{X},\tilde{%
Z}><\tilde{Y},\tilde{W}>] \\
+[\frac{c}{4f^{2}}-\frac{(f^{\prime })^{2}}{f^{2}}+\frac{f^{\prime \prime }}{%
f}][ &<&\tilde{X},\tilde{Z}><\tilde{Y},\partial _{t}><\tilde{W},\partial
_{t}> \\
- &<&\tilde{Y},\tilde{Z}><\tilde{X},\partial _{t}><\tilde{W},\partial _{t}>+<%
\tilde{Y},\tilde{W}><\tilde{X},\partial _{t}><\tilde{Z},\partial _{t}> \\
- &<&\tilde{X},\tilde{W}><\tilde{Y},\partial _{t}><\tilde{Z},\partial _{t}>]
\\
+\frac{c}{4f^{2}}[ &<&\tilde{X},\phi \tilde{Z}><\phi \tilde{Y},\tilde{W}>-<%
\tilde{Y},\phi \tilde{Z},><\phi \tilde{X},\tilde{W}> \\
+2 &<&\tilde{X},\phi \tilde{Y}><\phi \tilde{Z},\tilde{W}>]
\end{eqnarray*}%
and $[\mathcal{K},\mathcal{K}]=0.$
\end{proposition}

\begin{remark}
In \cite{FURUHATA} Furuhata, Hasegawa,Okuyama an Sato introduced Kenmotsu
statistical manifolds. They proved that if $M$ has a holomorphic statistical
structure, $(N=%
\mathbb{R}
\times _{e^{t}}M,<,>,\phi ,\xi )$ is Kenmotsu manifold satisfying property $%
\mathcal{\tilde{K}}_{X}Y=\mathcal{K}_{X}Y,\mathcal{\tilde{K}}_{X}\xi =%
\mathcal{\tilde{K}}_{\xi }X=0,\mathcal{\tilde{K}}_{\xi }\xi =\lambda \xi $
then $N$ has a holomorphic statistical structure,where $\lambda \in
C^{\infty }(N).$e now give a new example of a statistical warped product
manifold.
\end{remark}

\begin{example}[\protect \cite{CMBS}]
We consider $(%
\mathbb{R}
^{2},\tilde{g}=dx^{2}+dy^{2})$ Euclid space and define the affine connection
by

\begin{eqnarray}
\tilde{\nabla}_{\frac{\partial }{\partial x}}^{2}\frac{\partial }{\partial x}
&=&\frac{\partial }{\partial y},\text{ }\tilde{\nabla}_{\frac{\partial }{%
\partial y}}^{2}\frac{\partial }{\partial y}=0,  \label{R2} \\
\tilde{\nabla}_{\frac{\partial }{\partial x}}^{2}\frac{\partial }{\partial y}
&=&\tilde{\nabla}_{\frac{\partial }{\partial y}}^{2}\frac{\partial }{%
\partial x}=\frac{\partial }{\partial x}.  \notag
\end{eqnarray}%
Then its conjugate $\tilde{\nabla}^{2\ast }$ is given as follows;%
\begin{eqnarray}
\tilde{\nabla}_{\frac{\partial }{\partial x}}^{2\ast }\frac{\partial }{%
\partial x} &=&-\frac{\partial }{\partial y},\text{ }\tilde{\nabla}_{\frac{%
\partial }{\partial y}}^{2\ast }\frac{\partial }{\partial y}=0,  \label{R1}
\\
\tilde{\nabla}_{\frac{\partial }{\partial x}}^{2\ast }\frac{\partial }{%
\partial y} &=&\tilde{\nabla}_{\frac{\partial }{\partial y}}^{2\ast }\frac{%
\partial }{\partial x}=-\frac{\partial }{\partial x}.  \notag
\end{eqnarray}%
Thus we can verify that $(%
\mathbb{R}
^{2},\tilde{\nabla}^{2},\tilde{g})$ \ is a statistical manifold of constant
curvature $-1$. An affine connection and its conjugate connection are
defined on $(%
\mathbb{R}
,dt^{2})$ Euclidean space as following%
\begin{equation*}
\tilde{\nabla}_{\frac{\partial }{\partial t}}^{1}\frac{\partial }{\partial t}%
=0,\tilde{\nabla}_{\frac{\partial }{\partial t}}^{1\ast }\frac{\partial }{%
\partial t}=-0
\end{equation*}%
On the other hand , $%
\mathbb{R}
\times _{e^{t}}%
\mathbb{R}
^{2},<,>=dt^{2}+e^{2t}(dx^{2}+dy^{2}))$ is a warped product model of
hypebolic space $(\tilde{H}^{3}=\{(x,y,z)\in 
\mathbb{R}
^{3}\mid z>0\},\tilde{g}_{\tilde{H}^{3}}=\frac{dx^{2}+dy^{2}+dz^{2}}{z^{2}})$
and it has natural Kenmotsu structure. We \ also have $(%
\mathbb{R}
\times _{e^{t}}%
\mathbb{R}
^{2},<,>)$ is a statistic manifold with following affine connection $\bar{%
\nabla}$ ;%
\begin{eqnarray*}
\bar{\nabla}_{\frac{\partial }{\partial t}}\frac{\partial }{\partial t} &=&0,%
\text{ }\bar{\nabla}_{\frac{\partial }{\partial t}}\frac{\partial }{\partial
x}=\frac{\partial }{\partial x},\bar{\nabla}_{\frac{\partial }{\partial t}}%
\frac{\partial }{\partial y}=\frac{\partial }{\partial y}, \\
\bar{\nabla}_{\frac{\partial }{\partial x}}\frac{\partial }{\partial t} &=&%
\frac{\partial }{\partial x},\text{ }\bar{\nabla}_{\frac{\partial }{\partial
x}}\frac{\partial }{\partial x}=\frac{\partial }{\partial y}-e^{2t}\frac{%
\partial }{\partial t},\bar{\nabla}_{\frac{\partial }{\partial x}}\frac{%
\partial }{\partial y}=\frac{\partial }{\partial x}, \\
\bar{\nabla}_{\frac{\partial }{\partial y}}\frac{\partial }{\partial t} &=&%
\frac{\partial }{\partial y},\text{ }\bar{\nabla}_{\frac{\partial }{\partial
y}}\frac{\partial }{\partial x}=\frac{\partial }{\partial x},\bar{\nabla}_{%
\frac{\partial }{\partial y}}\frac{\partial }{\partial y}=-e^{2t}\frac{%
\partial }{\partial t}.
\end{eqnarray*}
\end{example}

\section{Generalized Wintgen Inequality for almost $(-\frac{f^{\prime }(t)}{%
f(t)})-$Kenmotsu statistical manifold}

Let $\bar{M}^{m}$ be a complex $m$- dimensional (real 2m dimensional) almost
Hermitian manifold with Hermitian metric $g_{\bar{M}}$ and almost complex
structure $J$ and $N^{n}$ be a Riemannian manifold with Riemannian metric $%
g_{N}.$ If $J$($T_{p}N)\subset T_{p}^{\perp }N,$at any point $p\in N,$then
is called totally real submanifold. In particular, a toatally real
submanifold of maximum dimension is called a Lagrangian submanifold.

Let$M^{n}$ be submanifold of $\tilde{M}^{2m+1}.$ $\phi $ maps any tangent
space of $M^{n}$ into the normal space, that is, $\phi (T_{p}M^{n})\subset
T_{p}^{\perp }\tilde{M}^{2m+1}$, for every $p\in M^{n}$ then $M^{n}$ is
called anti invarant submanifold. If $\dim (\tilde{M})=2\dim (M)+1$ and $\xi
_{p}$ is orthogonal to $T_{p}M$ for all $p\in M^{n}$ then $M^{n}$ is called
Legendre submanifold.

I. Mihai,\cite{ion},\cite{ion1} obtained the DDVV inequality,also known as
generalized Wintgen inequality for Lagrangian submanifold of a complex space
form $\bar{M}^{m}(4c)$ and Legendrian submanifolds in Sasakian space forms, 
\begin{equation*}
(\rho ^{\perp })^{2}\leq (\left \Vert H\right \Vert ^{2}-\rho +c)^{2}+\frac{4%
}{n(n-1)}(\rho -c)+\frac{2c^{2}}{n(n-1)},
\end{equation*}%
\begin{equation*}
(\rho ^{\perp })^{2}\leq (\left \Vert H\right \Vert ^{2}-\rho +c)^{2}+\frac{4%
}{n(n-1)}(\rho -\frac{c+3}{4})\frac{c-1}{4}+\frac{(c-1)^{2}}{8n(n-1)},
\end{equation*}%
respectively.

In \cite{BASJ}, the following theorem is proved

\begin{theorem}[\protect \cite{BASJ}]
Let $M^{n}$ be a Lagrangian submanifold of a holomorphic statistical space
form $\bar{M}^{m}(c)$. Then%
\begin{equation*}
(\rho ^{\perp })^{2}\geq \frac{c}{n(n-1)}(\rho -\frac{c}{4})+\frac{c}{%
(n-1)^{2}}[g(H^{\ast },H)-\left \Vert H\right \Vert \left \Vert H^{\ast
}\right \Vert ].
\end{equation*}
\end{theorem}

Now, we will prove Generalized Wintgen Inequality for almost $(-\frac{%
f^{\prime }(t)}{f(t)})-$Kenmotsu statistical manifold.

\begin{theorem}
Let $(%
\mathbb{R}
,dt,^{%
\mathbb{R}
}\nabla )$ be trival statistical manifold and\ $N(c)$ be a holomorphic
statistical space form. If $M^{n}$ is a Legendrian submanifold of the
statistical warped product manifold $\tilde{M}=%
\mathbb{R}
\times _{f}N(c)$ then we have 
\begin{eqnarray*}
\rho ^{\perp }{}^{\nabla ,\nabla ^{\ast }} &\leq &2\rho ^{\nabla ,\nabla
^{\ast }}-8\rho ^{0}+\frac{1}{4f^{2}}(2f\mid c\mid -c+4(f^{\prime })^{2}) \\
+4 &\parallel &H^{0}\parallel ^{2}+\parallel H\parallel ^{2}+\parallel
H^{\ast }\parallel ^{2}
\end{eqnarray*}
\end{theorem}

\begin{proof}
Let $M^{n}$ be an $n$-dimensional Legendrian real submanifold of a $2n+1$%
-dimensional almost $(-\frac{f^{\prime }(t)}{f(t)})-$Kenmotsu statistical
manifold $\tilde{M}=%
\mathbb{R}
\times _{f}N(c)$ and $\{e_{1},e_{2},...,e_{n}\}$ an orthonormal frame on $%
M^{n}$ and $\{e_{n+1}=\phi e_{1},e_{n+2}=\phi e_{2},...,e_{2n}=\phi
e_{n},e_{2n+1}=\xi \}$ an orthonormal frame in normal bundle $T^{\perp
}M^{n} $, respectively. By Proposition \ref{PRP} \ and (\ref{GAUSSEQU}), we
have 
\begin{eqnarray}
g_{M}(R(X,Y)Z,W) &=&<\tilde{R}(X,Y)Z,W>+<h^{\ast
}(X,W),h(Y,Z)>-<h(X,Z),h^{\ast }(Y,W)>  \notag \\
&=&[\frac{c}{4f^{2}}-\frac{(f^{\prime })^{2}}{f^{2}}][<Y,Z><X,W>-<X,Z><Y,W>]
\label{EQUA1} \\
+ &<&h^{\ast }(X,W),h(Y,Z)>-<h(X,Z),h^{\ast }(Y,W)>.  \notag
\end{eqnarray}%
and 
\begin{eqnarray}
g_{M}(R^{\ast }(X,Y)Z,W) &=&[\frac{c}{4f^{2}}-\frac{(f^{\prime })^{2}}{f^{2}}%
][<Y,Z><X,W>-<X,Z><Y,W>]  \notag \\
+ &<&h(X,W),h^{\ast }(Y,Z)>-<h^{\ast }(X,Z),h(Y,W)>.  \label{EQUA2A}
\end{eqnarray}%
\textit{for }$X,Y,Z,W\in \Gamma (TM)$. Setting $\ X=e_{i}=W,Y=e_{j}=Z$ in (%
\ref{EQUA1}) and (\ref{EQUA2A}), we have%
\begin{eqnarray}
g_{M}(R(e_{i},e_{j})e_{i},e_{j}) &=&(\frac{c}{4f^{2}}-\frac{(f^{\prime })^{2}%
}{f^{2}})(<e_{j},e_{j}><e_{i},e_{i}>-<e_{i},e_{j}><e_{i},e_{j}>)  \notag \\
+ &<&h^{\ast }(e_{i},e_{i}),h(e_{j},e_{j})>-<h(e_{i},e_{j}),h^{\ast
}(e_{i},e_{j})>  \label{EQUA3A}
\end{eqnarray}%
and 
\begin{eqnarray}
g_{M}(R^{\ast }(e_{i},e_{j})e_{i},e_{j}) &=&(\frac{c}{4f^{2}}-\frac{%
(f^{\prime })^{2}}{f^{2}}%
)(<e_{j},e_{j}><e_{i},e_{i}>-<e_{i},e_{j}><e_{i},e_{j}>)  \notag \\
+ &<&h(e_{i},e_{i}),h^{\ast }(e_{j},e_{j})>-<h^{\ast
}(e_{i},e_{j}),h(e_{i},e_{j})>  \label{EQUA4A}
\end{eqnarray}

Using (\ref{EQUA1}) in (\ref{RICCI}),\ we have 
\begin{eqnarray}
&<&(R^{\bot }(X,Y)U,V>=\frac{c}{4f^{2}}(-<\phi X,U><\phi Y,V>  \label{RICCI2}
\\
+ &<&\phi X,V><\phi Y,U>)+g_{M}([A_{U}^{\ast },A_{V}]X,Y),  \notag
\end{eqnarray}%
If we make use of the equality (\ref{EQUA2A}) in (\ref{DRICCI})I, we obtain 
\begin{eqnarray}
&<&(R^{\ast \bot }(X,Y)\xi ,\eta >=\frac{c}{4f^{2}}(-<\phi X,U><\phi Y,V>
\label{DRICCI2} \\
+ &<&\phi X,V><\phi Y,U>)+g_{M}([A_{\xi },A_{\eta }^{\ast }]X,Y).  \notag
\end{eqnarray}%
Since $<\tilde{R}(X,Y)Z,W>$ is not skew-symmetric relative to $Z$ and $W$.
Then the sectional curvature on $\tilde{M}$ \ can not be defined. But $%
<R(X,Y)Z,W>+<R^{\ast }(X,Y)Z,W)>$ is skew-symmetric relative to $Z$ and $W.$
So the sectional curvature $K^{\nabla ,\nabla ^{\ast }}$ is defined by 
\begin{equation*}
K^{\nabla ,\nabla ^{\ast }}(X\wedge Y)=\frac{1}{2}[<R(X,Y)Y,X>+<R^{\ast
}(X,Y)Y,X)>],
\end{equation*}%
for any orthonormal vectors $X,Y,\in T_{p}M$, $p\in M,$ (see \cite{AYMM}).

In \cite{AYMM}, the normalized scalar curvature $\rho ^{\nabla ,\nabla
^{\ast }}$and the normalized normal scalar curvature $\rho ^{\perp
}{}^{\nabla ,\nabla ^{\ast }}$are respectively defined by 
\begin{eqnarray*}
\rho ^{\nabla ,\nabla ^{\ast }} &=&\frac{2}{n(n-1)}\dsum \limits_{1\leq
i<j\leq n}K^{\nabla ,\nabla ^{\ast }}(e_{i}\wedge e_{j}) \\
&=&\frac{2}{n(n-1)}\dsum \limits_{1\leq i<j\leq
n}<R(e_{i},e_{j})e_{j},e_{i}>+<R^{\ast }(e_{i},e_{j})e_{j},e_{i}>]
\end{eqnarray*}%
and 
\begin{equation*}
\rho ^{\perp }{}^{\nabla ,\nabla ^{\ast }}=\frac{1}{n(n-1)}\left \{ \dsum
\limits_{n+1\leq r<s\leq 2n}\dsum \limits_{1\leq i<j\leq n}\left[ <R^{\perp
}(e_{i},e_{j})e_{r},e_{s}>+<R^{\ast \perp }(e_{i},e_{j})e_{r},e_{s}>\right]
^{2}\right \} ^{1/2},
\end{equation*}%
where $\{e_{1},...,e_{n}\}$ and $\{e_{n+1}=\phi e_{1},...,e_{2n}=\phi
e_{n},e_{2n+1}=\xi \}$ are respectively orthonormal basis of $T_{p}M$ and $%
T_{p}^{\perp }M$ for $p\in M$. Due to the equations (\ref{EQUA2A}) and (\ref%
{EQUA3A}) ,we obtain%
\begin{eqnarray*}
\rho ^{\nabla ,\nabla ^{\ast }} &=&\frac{1}{2n(n-1)}\dsum \limits_{i\neq j}[(%
\frac{c}{4f^{2}}-\frac{(f^{\prime })^{2}}{f^{2}})+<h^{\ast
}(e_{i},e_{i}),h(e_{j},e_{j})>-<h(e_{i},e_{j}),h^{\ast }(e_{i},e_{j})> \\
+(\frac{c}{4f^{2}}-\frac{(f^{\prime })^{2}}{f^{2}})+
&<&h(e_{i},e_{i}),h^{\ast }(e_{j},e_{j})>-<h^{\ast
}(e_{i},e_{j}),h(e_{i},e_{j})> \\
&=&(\frac{c}{4f^{2}}-\frac{(f^{\prime })^{2}}{f^{2}})+\frac{1}{2n(n-1)}\dsum
\limits_{i\neq j}[<h^{\ast
}(e_{i},e_{i}),h(e_{j},e_{j})>+<h(e_{i},e_{i}),h^{\ast }(e_{j},e_{j})> \\
-2 &<&h(e_{i},e_{j}),h^{\ast }(e_{i},e_{j})> \\
&=&(\frac{c}{4f^{2}}-\frac{(f^{\prime })^{2}}{f^{2}})+\frac{1}{2n(n-1)}\dsum
\limits_{i\neq j}[<h(e_{i},e_{i})+h^{\ast }(e_{i},e_{i}),h^{\ast
}(e_{j},e_{j})+h(e_{j},e_{j})> \\
- &<&h(e_{i},e_{i}),h(e_{j},e_{j})>-<h^{\ast }(e_{j},e_{j}),h^{\ast
}(e_{j},e_{j})> \\
-( &<&h(e_{i},e_{j})+h^{\ast }(e_{i},e_{j}),h(e_{i},e_{j})+h^{\ast
}(e_{i},e_{j}>-<h(e_{i},e_{j}),h(e_{i},e_{j})> \\
- &<&h^{\ast }(e_{i},e_{j}),h^{\ast }(e_{i},e_{j})>)
\end{eqnarray*}%
Because of $2h^{0}=h+h^{\ast }$ and $2H^{0}=H+H^{\ast },$ we thus get%
\begin{eqnarray*}
\rho ^{\nabla ,\nabla ^{\ast }} &=&(\frac{c}{4f^{2}}-\frac{(f^{\prime })^{2}%
}{f^{2}})+\frac{1}{2n(n-1)}\dsum \limits_{i\neq
j}[4<h^{0}(e_{i},e_{i}),h^{0}(e_{j},e_{j}))> \\
- &<&h(e_{i},e_{i}),h(e_{j},e_{j})>-<h^{\ast }(e_{j},e_{j}),h^{\ast
}(e_{j},e_{j})> \\
-(4 &<&h^{0}(e_{i},e_{j}),h^{0}(e_{i},e_{j})>-<h(e_{i},e_{j}),h(e_{i},e_{j})>
\\
- &<&h^{\ast }(e_{i},e_{j}),h^{\ast }(e_{i},e_{j})>)
\end{eqnarray*}%
\begin{eqnarray*}
\rho ^{\nabla ,\nabla ^{\ast }} &=&(\frac{c}{4f^{2}}-\frac{(f^{\prime })^{2}%
}{f^{2}})+\frac{1}{2n(n-1)}[4n^{2}\parallel H^{0}\parallel
^{2}-n^{2}\parallel H\parallel ^{2}-n^{2}\parallel H^{\ast }\parallel ^{2} \\
-(4 &\parallel &h^{0}\parallel ^{2}-\parallel h\parallel ^{2}-\parallel
h^{\ast }\parallel ^{2}].
\end{eqnarray*}%
Denote $\tau ^{0}=h^{0}-H^{0}g$, $\tau =h-Hg$ and $\tau ^{\ast }=h^{\ast
}-H^{\ast }g$ the traceless part of second fundamental forms. Then we find $%
\parallel \tau ^{0}\parallel ^{2}=\parallel h^{0}\parallel
^{2}-n^{2}\parallel H^{0}\parallel ^{2}$, $\parallel \tau \parallel
^{2}=\parallel h\parallel ^{2}-n^{2}\parallel H\parallel ^{2}$ and $%
\parallel \tau ^{\ast }\parallel ^{2}=\parallel h^{\ast }\parallel
^{2}-n^{2}\parallel H^{\ast }\parallel ^{2}$. Thus, we get%
\begin{eqnarray*}
\rho ^{\nabla ,\nabla ^{\ast }} &=&(\frac{c}{4f^{2}}-\frac{(f^{\prime })^{2}%
}{f^{2}})+\frac{1}{2n(n-1)}[4n^{2}\parallel H^{0}\parallel
^{2}-n^{2}\parallel H\parallel ^{2}-n^{2}\parallel H^{\ast }\parallel ^{2} \\
-(4 &\parallel &\tau ^{0}\parallel ^{2}+4n\parallel H^{0}\parallel
^{2}-\parallel \tau \parallel ^{2}-n\parallel H\parallel ^{2}-\parallel \tau
^{\ast }\parallel ^{2}-n\parallel H^{\ast }\parallel ^{2}).
\end{eqnarray*}%
This relation gives rise to 
\begin{eqnarray}
\rho ^{\nabla ,\nabla ^{\ast }} &=&\frac{c}{4f^{2}}-\frac{(f^{\prime })^{2}}{%
f^{2}}  \notag \\
+2 &\parallel &H^{0}\parallel ^{2}-\frac{2}{n(n-1)}\parallel \tau
^{0}\parallel ^{2}  \notag \\
-\frac{1}{2} &\parallel &H\parallel ^{2}+\frac{1}{2n(n-1)}\parallel \tau
\parallel ^{2}  \label{RODELTATDELTASTAR} \\
-\frac{1}{2} &\parallel &H^{\ast }\parallel ^{2}+\frac{1}{2n(n-1)}\parallel
\tau ^{\ast }\parallel ^{2}.  \notag
\end{eqnarray}%
From (\ref{RICCI2}) and (\ref{DRICCI2}), the normalized normal scalar
curvature satisfies 
\begin{equation*}
\rho ^{\perp }{}^{\nabla ,\nabla ^{\ast }}=\frac{1}{n(n-1)}\left \{ \dsum
\limits_{1\leq r<s\leq n+1}\dsum \limits_{1\leq i<j\leq n}\left[ 
\begin{array}{c}
g([A_{e_{n+r}}^{\ast
},A_{e_{n+s}}]e_{i},e_{j})+g([A_{e_{n+r}},A_{e_{n+s}}^{\ast }]e_{i},e_{j})
\\ 
+\frac{2c}{4f^{2}}(-<\phi e_{i},e_{n+r}><\phi e_{j},e_{n+s}>+ \\ 
+<\phi e_{i},e_{n+s}><\phi e_{j},e_{n+r}>))%
\end{array}%
\right] ^{2}\right \} ^{1/2}
\end{equation*}%
\begin{equation}
=\frac{1}{n(n-1)}\left \{ \dsum \limits_{1\leq r<s\leq n+1}\dsum
\limits_{1\leq i<j\leq n}\left[ 
\begin{array}{c}
g([A_{e_{n+r}}^{\ast },A_{e_{n+s}}]e_{i},e_{j}) \\ 
+g([A_{e_{n+r}},A_{e_{n+s}}^{\ast }]e_{i},e_{j}) \\ 
-\frac{2c}{4f^{2}}(\delta _{ir}\delta _{js}-\delta _{is}\delta _{jr})%
\end{array}%
\right] ^{2}\right \} ^{1/2}  \label{MU1}
\end{equation}%
.By Proposition \ref{TO1} and the equations \ref{GAUSS1*}, \ref{GAUSS2},we
have 
\begin{equation}
A_{\xi }X=A_{\xi }^{\ast }X=-\frac{f^{\prime }(t)}{f(t)}X\text{ .}
\label{ALMU}
\end{equation}%
So we have 
\begin{eqnarray}
g([A_{\xi }^{\ast },A_{e_{n+s}}]e_{i},e_{j}) &=&g(A_{\xi }^{\ast
}A_{e_{n+s}}e_{i},e_{j})-g(A_{e_{n+s}}A_{\xi }^{\ast }e_{i},e_{j})  \notag \\
&&\overset{(\ref{ALMU})}{=}-\frac{f^{\prime }(t)}{f(t)}%
g(A_{e_{n+s}}e_{i},e_{j})+\frac{f^{\prime }(t)}{f(t)}%
g(A_{e_{n+s}}e_{i},e_{j})  \label{ALMU2} \\
&=&0  \notag
\end{eqnarray}%
and by using similar calculation we obtain 
\begin{equation}
([A_{e_{n+r}},A_{e_{n+s}}^{\ast }]e_{i},e_{j})=0.  \label{ALMU3}
\end{equation}%
On the other hand \ we recall 
\begin{equation}
<\phi X,\xi >=0.  \label{ALMU1}
\end{equation}%
Using the equations \ref{ALMU2}, \ref{ALMU3} and \ref{ALMU1} in \ref{MU1},
we find that 
\begin{equation*}
\rho ^{\perp }{}^{\nabla ,\nabla ^{\ast }}=\frac{1}{n(n-1)}\left \{ \dsum
\limits_{1\leq r<s\leq n+1}\dsum \limits_{1\leq i<j\leq n}\left[ 
\begin{array}{c}
g([A_{e_{n+r}}^{\ast
},A_{e_{n+s}}]e_{i},e_{j})+g([A_{e_{n+r}},A_{e_{n+s}}^{\ast }]e_{i},e_{j})
\\ 
-\frac{2c}{4f^{2}}(\delta _{ir}\delta _{js}-\delta _{is}\delta _{jr})%
\end{array}%
\right] ^{2}\right \} ^{1/2}
\end{equation*}%
which is equivalent 
\begin{equation}
\rho ^{\perp }{}^{\nabla ,\nabla ^{\ast }}=\frac{1}{n(n-1)}\left \{ \dsum
\limits_{1\leq r<s\leq n}\dsum \limits_{1\leq i<j\leq n}\left[ 
\begin{array}{c}
4g([A_{e_{n+r}}^{0},A_{e_{n+s}}^{0}]e_{i},e_{j})+g([A_{e_{n+r}},A_{e_{n+s}}]e_{i},e_{j})
\\ 
+g([A_{e_{n+r}}^{\ast },A_{n+s_{s}}^{\ast }]e_{i},e_{j})-\frac{2c}{4f^{2}}%
(\delta _{ir}\delta _{js}-\delta _{is}\delta _{jr})%
\end{array}%
\right] ^{2}\right \} ^{1/2},  \label{ALMU4}
\end{equation}%
where $2A_{\xi _{r}}^{0}=A_{\xi _{r}}+A_{\xi _{r}}^{\ast }$. By the
Cauchy--Schwarz inequality, we have the algebraic inequality%
\begin{equation}
(\lambda +\mu +\nu +w)^{2}\leq 4(\lambda ^{2}+\mu ^{2}+\nu
^{2}+w^{2}),\forall \lambda ,\mu ,\nu \in 
\mathbb{R}
.  \label{ALG}
\end{equation}%
We obtain from (\ref{ALG}) that 
\begin{eqnarray*}
\rho ^{\perp }{}^{\nabla ,\nabla ^{\ast }} &\leq &\frac{2}{n(n-1)}\left \{ 
\begin{array}{c}
\dsum \limits_{1\leq r<s\leq n}(\dsum \limits_{1\leq i<j\leq
m}(16g([A_{e_{n+r}}^{0},A_{e_{n+s}}^{0}]e_{i},e_{j})^{2}+g([A_{e_{n+r}},A_{e_{n+s}}]e_{i},e_{j})^{2}
\\ 
+g([A_{e_{n+r}}^{\ast },A_{n+s_{s}}^{\ast }]e_{i},e_{j})^{2}+\frac{c^{2}}{%
4f^{2}}(\delta _{ir}\delta _{js}-\delta _{is}\delta _{jr})^{2}%
\end{array}%
\right \} ^{1/2} \\
&\leq &\frac{2}{n(n-1)}\left \{ 
\begin{array}{c}
\frac{c^{2}}{4f^{2}}n^{2}(n-1)^{2}+\frac{1}{4}\dsum
\limits_{r,s=1}^{n}(\dsum \limits_{i,j=1}^{m}16g([A_{\xi _{r}}^{0},A_{\xi
_{s}}^{0}]e_{i},e_{j})^{2}+g([A_{\xi _{r}},A_{\xi _{s}}]e_{i},e_{j})^{2} \\ 
+g([A_{\xi _{r}}^{\ast },A_{\xi _{s}}^{\ast }]e_{i},e_{j}))^{2}%
\end{array}%
\right \} ^{1/2} \\
&\leq &\frac{2}{n(n-1)}\left \{ \frac{c^{2}}{4f^{2}}n^{2}(n-1)^{2}+\frac{1}{4%
}\dsum \limits_{r,s=1}^{n}(16\Vert \lbrack A_{\xi _{r}}^{0},A_{\xi
_{s}}^{0}\Vert ^{2}+\Vert \lbrack A_{\xi _{r}},A_{\xi _{s}}]\Vert ^{2}+\Vert
\lbrack A_{\xi _{r}}^{\ast },A_{\xi _{s}}^{\ast }]\Vert ^{2})\right \}
^{1/2}.
\end{eqnarray*}%
Now we define sets $\{S_{1}^{0},...,S_{n}^{0}\},$ $\{S_{1},...,S_{n}\},%
\{S_{1}^{\ast },...,S_{n}^{\ast }\}$of symmetric with trace zero operators
on $T_{p}M$ by 
\begin{eqnarray*}
&<&S_{\alpha }^{0}X,Y>=<\tau ^{0}(X,Y),\xi _{\alpha }>, \\
&<&S_{\alpha }X,Y>=<\tau (X,Y),\xi _{\alpha }>, \\
&<&S_{\alpha }^{\ast }X,Y>=<\tau ^{\ast }(X,Y),\xi _{\alpha }>
\end{eqnarray*}%
for all $X,Y,\in T_{p}M$, $p\in M$. Clearly, we obtain%
\begin{eqnarray*}
S_{\alpha }^{0} &=&A_{\xi _{\alpha }}^{0}-<H^{0},\xi _{\alpha }>I, \\
S_{\alpha } &=&A_{\xi _{\alpha }}-<H,\xi _{\alpha }>I, \\
S_{\alpha }^{\ast } &=&A_{\xi _{\alpha }}^{\ast }-<H^{\ast },\xi _{\alpha }>I
\end{eqnarray*}%
and%
\begin{eqnarray*}
\lbrack S_{\alpha }^{0},S_{\beta }^{0}] &=&[A_{\xi _{\alpha }}^{0},A_{\xi
_{\beta }}^{0}], \\
\lbrack S_{\alpha },S_{\beta }] &=&[A_{\xi _{\alpha }},A_{\xi _{\beta }}], \\
\lbrack S_{\alpha }^{\ast },S_{\beta }^{\ast }] &=&[A_{\xi _{\alpha }}^{\ast
},A_{\xi _{\beta }}^{\ast }].
\end{eqnarray*}%
Therefore, it is clear that 
\begin{equation}
\rho ^{\perp }{}^{\nabla ,\nabla ^{\ast }}\leq \frac{2}{n(n-1)}\left \{ 
\frac{c^{2}}{4f^{2}}n^{2}(n-1)^{2}+\frac{1}{4}\dsum
\limits_{r,s=1}^{n}(16\Vert \lbrack S_{r}^{0},S_{s}^{0}]\Vert ^{2}+\Vert
\lbrack S_{r},S_{s}]\Vert ^{2}+\Vert \lbrack S_{r}^{\ast },S_{s}^{\ast
}]\Vert ^{2})\right \} ^{1/2}.  \label{INEQUALITY}
\end{equation}%
In \cite{lu} Lu \ proved following theorem

\begin{theorem}[ \protect \cite{lu}]
\label{Lu}Let $M^{m}$ be Riemannian submanifold of Riemannian space form $%
\tilde{M}^{m+n}(c)$. For every set $\{B_{1},...,B_{n}\}$of symmetric $%
(n\times n)$-matrices with trace zero the following inequality holds:%
\begin{equation*}
\dsum \limits_{\alpha ,\beta =1}^{n}\Vert \lbrack B_{\alpha },B_{\beta
}]\Vert ^{2}\leq (\dsum \limits_{\alpha =1}^{n}\Vert B_{\alpha }\Vert
^{2})^{2}
\end{equation*}
\end{theorem}

By Theorem \ref{Lu}, (\ref{INEQUALITY}) can be written as%
\begin{eqnarray}
\rho ^{\perp }{}^{\nabla ,\nabla ^{\ast }} &\leq &\frac{\mid c\mid }{2f}+%
\frac{4}{n(n-1)}\dsum \limits_{r=1}^{n}\Vert S_{r}^{0}\Vert ^{2}+\frac{1}{%
n(n-1)}\dsum \limits_{r=1}^{n}\Vert \lbrack S_{r}\Vert ^{2}+\frac{1}{n(n-1)}%
\dsum \limits_{r=1}^{n}\Vert \lbrack S_{r}^{\ast }\Vert ^{2}  \notag \\
&\leq &\frac{\mid c\mid }{2f}+\frac{4}{n(n-1)}\parallel \tau ^{0}\parallel
^{2}+\frac{1}{n(n-1)}\parallel \tau \parallel ^{2}+\frac{1}{n(n-1)}\parallel
\tau ^{\ast }\parallel ^{2}.  \label{ROORTHOGONAL}
\end{eqnarray}%
Using (\ref{RODELTATDELTASTAR}) in (\ref{ROORTHOGONAL}), we get%
\begin{eqnarray}
\rho ^{\perp }{}^{\nabla ,\nabla ^{\ast }} &\leq &\frac{\mid c\mid }{2f}+%
\frac{8}{n(n-1)}\parallel \tau ^{0}\parallel ^{2}+2\rho ^{\nabla ,\nabla
^{\ast }}-\frac{2c}{4f^{2}}+\frac{2(f^{\prime })^{2}}{f^{2}}
\label{ROORTHOGONAL2} \\
-4 &\parallel &H^{0}\parallel ^{2}+\parallel H\parallel ^{2}+\parallel
H^{\ast }\parallel ^{2}  \notag
\end{eqnarray}%
The other hand normalized scalar curvature $\rho ^{0}$ of $M^{m}$ \ with
respect to Levi-civita connection $\nabla ^{0\text{ }}$can be obtained as%
\begin{equation}
\rho ^{0}=(\frac{c}{4f^{2}}-\frac{(f^{\prime })^{2}}{f^{2}})+\frac{1}{n(n-1)}%
[n^{2}\parallel H^{0}\parallel ^{2}-\parallel h^{0}\parallel ^{2}]
\label{ROZEROO}
\end{equation}%
(see \cite{ROTH}).

Now, If we set $\parallel \tau ^{0}\parallel ^{2}=\parallel h^{0}\parallel
^{2}-n\parallel H^{0}\parallel ^{2}$in (\ref{ROZEROO}) then we get%
\begin{equation}
\rho ^{0}=(\frac{c}{4f^{2}}-\frac{(f^{\prime })^{2}}{f^{2}})+\parallel
H^{0}\parallel ^{2}-\frac{1}{n(n-1)}\parallel \tau ^{0}\parallel ^{2}.
\label{ROZERO3}
\end{equation}%
In view of the equations (\ref{ROORTHOGONAL2}) and (\ref{ROZEROO}) we have%
\begin{eqnarray*}
\rho ^{\perp }{}^{\nabla ,\nabla ^{\ast }} &\leq &2\rho ^{\nabla ,\nabla
^{\ast }}-8\rho ^{0}+\frac{1}{4f^{2}}(2f\mid c\mid -c+4(f^{\prime })^{2}) \\
+4 &\parallel &H^{0}\parallel ^{2}+\parallel H\parallel ^{2}+\parallel
H^{\ast }\parallel ^{2}
\end{eqnarray*}%
which completes proof.
\end{proof}

\begin{corollary}
Let $(%
\mathbb{R}
,dt,^{%
\mathbb{R}
}\nabla )$ be trival statistical manifold and\ $N(c)$ be a holomorphic
statistical space form. If $M^{n}$ be a Legendrian submanifold of the
statistical Kenmotsu manifold $%
\mathbb{R}
\times _{e^{t}}%
\mathbb{C}
^{n}$, then we get 
\begin{equation*}
\rho ^{\perp }{}^{\nabla ,\nabla ^{\ast }}\leq 2\rho ^{\nabla ,\nabla ^{\ast
}}-8\rho ^{0}+4\parallel H^{0}\parallel ^{2}+\parallel H\parallel
^{2}+\parallel H^{\ast }\parallel ^{2}+1.
\end{equation*}%
In this case $%
\mathbb{R}
\times _{e^{t}}%
\mathbb{C}
^{n}$ is locally isometric to the hyperbolic space $H^{2n+1}(-1)$.
\end{corollary}

\begin{corollary}
Let $(%
\mathbb{R}
,dt,^{%
\mathbb{R}
}\nabla )$ be trival statistical manifold and\ $N(c)$ be a holomorphic
statistical space form.If $M^{n}$ be a Legendrian submanifold of the
statistical cosymplectic manifold $%
\mathbb{R}
\times N(c)$, then we have%
\begin{equation*}
\rho ^{\perp }{}^{\nabla ,\nabla ^{\ast }}\leq 2\rho ^{\nabla ,\nabla ^{\ast
}}-8\rho ^{0}+4\parallel H^{0}\parallel ^{2}+\parallel H\parallel
^{2}+\parallel H^{\ast }\parallel ^{2}+\frac{1}{4}(2\mid c\mid -c)
\end{equation*}
\end{corollary}

\end{document}